\documentclass[runningheads]{llncs}

\usepackage{amsmath,amssymb}
\usepackage{hyperref,comment}
\usepackage{tikz}

\newcommand{\Z}{\mathbb{Z}}
\newcommand{\N}{\mathbb{N}}

\newcommand{\realm}{\mathfrak{D}}

\newcommand{\lang}{\mathcal{L}}
\newcommand{\glimset}{\tilde \omega}
\newcommand{\sftap}{\mathcal{S}}

\begin{document}

\title{Complexity of Generic Limit Sets of Cellular Automata}
\author{Ilkka T\"orm\"a\thanks{Research supported by Academy of Finland grant 295095}}
\institute{Department of Mathematics and Statistics \\ University of Turku \\ Finland \\ \email{iatorm@utu.fi}}

\maketitle

\begin{abstract}
The generic limit set of a topological dynamical system is the smallest closed subset of the phase space that has a comeager realm of attraction.
It intuitively captures the asymptotic dynamics of almost all initial conditions.
It was defined by Milnor and studied in the context of cellular automata, whose generic limit sets are subshifts, by Djenaoui and Guillon.
In this article we study the structural and computational restrictions that apply to generic limit sets of cellular automata.
As our main result, we show that the language of a generic limit set can be at most $\Sigma^0_3$-hard, and lower in various special cases.
We also prove a structural restriction on generic limit sets with a global period.

\end{abstract}

\section{Introduction}

One-dimensional cellular automata (CA for short) are discrete dynamical systems that act on the set $A^\Z$ of bi-infinite sequences by a local rule that is applied synchronously at every coordinate.
They can be used to model physical and biological phenomena, as well as massively parallel computation.

The limit set of a topological dynamical system $(X, T)$ consists of those points that can be seen arbitrarily late in its evolution.
Limit sets of cellular automata have been studied by various authors from the computational (e.g.~\cite{CuPaYu89,Ka92,BoCeVa14}) and structural (e.g.~\cite{Ma95,BaGuKa11}) points of view.
In~\cite{Mi85}, Milnor defined the likely and generic limit sets of a dynamical system.
The likely limit set associated to an invariant probability measure $\mu$ on $X$ is the smallest closed subset $C \subset X$ such that for $\mu$-almost every $x \in X$, all limit points of $(T^n(x))_{n \in \N}$ are in $C$.
The generic variant is a purely topological notion that replaces ``for $\mu$-almost every $x \in X$'' with ``for every $x$ in a comeager subset of $X$''.
As far as we know, generic limit sets have been studied relatively little in dynamical systems theory.

In~\cite{DjGu19}, Djenaoui and Guillon studied the generic limit sets of dynamical systems in general and \emph{directed cellular automata} (arbitrary paths in spacetime diagrams of CA) in particular.
They related dynamical properties of a given CA to the structure of its generic limit set in different directions and its relation to the set of equicontinuity points and the limit set.
For example, they proved that the generic limit set of an almost equicontinuous CA is exactly the closure of the asymptotic set of its equicontinuity points.
They also provide a combinatorial characterization of the generic limit set of a CA, which allows us to study its descriptional and structural complexity and carry out complex constructions.
This point of view was not present in~\cite{DjGu19}, where relatively simple examples of generic limit sets were provided to highlight the main classification results.

As our main result, we prove that the language of a generic limit set of a CA is always $\Sigma^0_3$, and present an example which is complete for this class.
If the generic limit set is minimal, then this bound cannot be attained, since its language is $\Sigma^0_2$.
We also prove that the dynamics of the CA on its generic limit set must be nontrivial in complex instances: if the CA is eventually periodic or strictly one-sided on the generic limit set, its language is $\Sigma^0_1$ or $\Pi^0_2$, respectively, and if the restriction is a shift map, then the generic limit set is chain-transitive.
These restrictions are proved by constructing ``semi-blocking words'' that restrict the flow of information.
Finally, we present a structural restriction for generic limit sets: if they consist of a finite number of two-way chain components for the shift map, then they cannot have a global period.


\section{Definitions}

Let $X$ be a topological space.
A subset of $X$ is comeager if it contains an intersection of countably many dense open sets. 


A dynamical system is a pair $(X, f)$ where $X$ is a compact metric space and $f : X \to X$ is a continuous function.
We say $(X, f)$ has trivial dynamics if $f = \mathrm{id}_X$.
The limit set of $f$ is $\Omega_f = \bigcap_{t \in \N} f^t(X)$.
For $x \in X$, we define $\omega(x)$ as the set of limit points of the forward orbit $(f^t(x))_{t \in \N}$, and $\omega(Y) = \bigcup_{y \in Y} \omega(y)$ for $Y \subset X$.
The realm (of attraction) of a subset $Y \subset X$ is $\realm(Y) = \{ x \in X \;|\; \omega(x) \subset Y \}$.
The generic limit set of $f$, denoted $\glimset(f)$, is the intersection of all closed subsets $C \subset X$ such that $\realm(C)$ is comeager; then $\glimset(f)$ itself has a comeager realm.

We consider one-dimensional cellular automata over a finite alphabet $A$.
The full shift $A^\Z$ is a compact metric space with the distance function $d(x,y) = \inf \{ 2^{-n} \;|\; x_{[-n,n]} = y_{[-n,n]} \}$, and the left shift map $\sigma : A^\Z \to A^\Z$, defined by $\sigma(x)_i = x_{i+1}$, is a homeomorphism.
The cylinder sets $[w]_i = \{ x \in A^\Z \;|\; x_{[i, i+|w|)} = w \}$ for $w \in A^*$ and $i \in \Z$ form a prebasis for the topology, and the clopen sets, which are the finite unions of cylinders, form a basis.
We denote $[w] = [w]_0$.
A subshift is a closed and $\sigma$-invariant set $X \subset A^\Z$.
Every subshift is defined by a set $F \subset A^*$ of forbidden words as $X = A^\Z \setminus \bigcup_{w \in F} \bigcup_{i \in \Z} [w]_i$, and if $F$ can be chosen finite, then $X$ is a shift of finite type (SFT).
The language of $X$ is defined as $\lang(X) = \{ w \in A^* \;|\; [w] \cap X \neq \emptyset \}$, and we denote $\lang_n(X) = \lang(X) \cap A^n$.
The order-$n$ SFT approximation of $X$ is the SFT $\sftap_n(X) \subset A^\Z$ defined by the forbidden patterns $A^n \setminus \lang_n(X)$.
We say $X$ is transitive if for all $u, v \in \lang(X)$ there exists $w \in A^*$ with $u w v \in \lang(X)$, and mixing if the length of $w$ can be chosen freely as long as it is large enough (depending on $u$ and $v$).
We say $X$ is chain transitive if each $\sftap_n(X)$ is transitive.
We say $X$ is minimal if it does not properly contain another subshift; this is equivalent to the condition that for every $w \in \lang(X)$ there exists $n \in \N$ such that $w$ occurs in each word of $\lang_n(X)$.

A morphism between dynamical systems $(X, f)$ and $(Y, g)$ is a continuous function $h : X \to Y$ with $h \circ f = g \circ h$.
If $h$ is surjective, $(Y, g)$ is a factor of $(X, f)$.
A cellular automaton is a morphism $f : (A^\Z, \sigma) \to (A^\Z, \sigma)$.
Equivalently, it is a function given by a local rule $F : A^{2r+1} \to A$ for some radius $r \in \N$ as $f(x)_i = F(x_{[i-r,i+r]})$.
The pair $(A^\Z, f)$ is a dynamical system.
Generic limit sets were defined by Milnor in \cite{Mi85} for general dynamical systems, and were first considered in the context of cellular automata in \cite{DjGu19}.

In this article, a Turing machine consists of a finite state set $Q$ with an initial state $q_0$ and a final state $q_f$, a tape alphabet $\Gamma$ that is used on a one-way infinite tape together with a special blank symbol $\bot \notin \Gamma$, and a transition rule $\delta$ that allows the machine to move on the tape and modify the tape cells and its internal state based on its current state and the contents of the tape cell it is on.
Turing machines can decide any computable language and compute any computable function in the standard way.

We give an overview of the arithmetical hierarchy.
A computable predicate over $\N$ is $\Pi^0_0$ and $\Sigma^0_0$.
If $\phi$ is a $\Pi^0_n$ predicate, then $\exists k_1 \cdots \exists k_m \phi$ is a $\Sigma^0_{n+1}$ formula, and conversely, if $\phi$ is $\Sigma^0_n$, then $\forall k_1 \cdots \forall k_m \phi$ is $\Pi^0_{n+1}$.
Subsets of $\N$ defined by these formulas are given the same classifications, and we extend them to all sets that are in a computable bijection with $\N$.
For these sets, we define $\Delta^0_n = \Pi^0_n \cap \Sigma^0_n$.
The computable sets form $\Delta^0_1$ and the computably enumerable sets form $\Sigma^0_1$.
A subshift is given the same classification as its language.

\section{Auxiliary results}

We begin with auxiliary results on generic limit sets of cellular automata that are used in several proofs.

\begin{lemma}[Proposition~4.11 in~\cite{DjGu19}]
\label{lem:Invariant}
Let $f$ be a CA.
Then $\glimset(f)$ is a nonempty $f$-invariant subshift.
\end{lemma}

The following result gives a combinatorial characterization for generic limit sets of cellular automata.

\begin{lemma}[Corollary of Remark~4.4 in \cite{DjGu19}]
\label{lem:CombChar}
Let $f$ be a CA on $A^\Z$.
A word $s \in A^*$ occurs in $\glimset(f)$ if and only if there exists a word $v \in A^*$ and $i \in \Z$ such that for all $u, w \in A^*$ there exist infinitely many $t \in \N$ with $f^t([u v w]_{i - |u|}) \cap [s] \neq \emptyset$.
\end{lemma}

We say that the word $v$ \emph{enables} $s$ for $f$.

  

\begin{lemma}
\label{lem:ForcingWords}
Let $f$ be a CA on $A^\Z$, let $n \in \N$, and let $[v]_i \subset A^\Z$ be a cylinder set.
Then there exists a cylinder set $[w]_j \subset [v]_i$ and $T \in \N$ such that for all $t \geq T$ we have $f^t([w]_j) \subset [\lang_n(\glimset(f))]$.
\end{lemma}

Words $w$ with the above property are called \emph{$\glimset(f)$-forcing}, since they force the word $f^t(x)_{[0, n)}$ to be valid in $\glimset(f)$ whenever $w$ occurs in $x$ at position $j$.
The result intuitively states that any word can be extended into a $\glimset(f)$-forcing word.

\begin{proof}
Denote $A^n \setminus \lang_n(\glimset(f)) = \{u_1, \ldots, u_k\}$.
Since $u_1$ does not occur in $\glimset(f)$, Lemma~\ref{lem:CombChar} applied to $[v]_i$ implies that there exist words $a_1, b_1 \in A^*$ and $T_1 \in \N$ such that $f^t([a_1 v b_1]_{i-|a_1|}) \cap [u_1] = \emptyset$ for all $t \geq T_1$.
For $u_2$ we find words $a_2, b_2 \in A^*$ and $T_2 \in \N$ such that $f^t([a_2 a_1 v b_1 b_2]_{i-|a_2 a_1|}) \cap [u_2] = \emptyset$ for all $t \geq T_2$.
Continuing like this, we obtain a word $w = a_k \cdots a_1 v b_1 \cdots b_k$, a position $j = i - |a_k \cdots a_1|$ and a number $T = \max(T_1, \ldots, T_k)$ that have the desired property.
\end{proof}

\begin{example}[Example~5.12 in~\cite{DjGu19}]
\label{ex:GenNilp}
Consider the minimum CA $f : \{0,1\}^\Z \to \{0,1\}^\Z$ defined by $f(x)_i = \min(x_i, x_{i+1})$.
We claim that $\glimset(f) = \{{}^\infty 0^\infty\}$.
Proving this directly from the definition is not difficult, but let us illustrate the use of Lemma~\ref{lem:CombChar}.
First, every word $0^n$ for $n \in \N$ is enabled by itself: for all $u, v \in \{0,1\}^*$ and $t \in \N$ we have $f^t([u 0^n v]_{-|u|}) \subset [0^n]$.
On the other hand, suppose $s \in \lang(\glimset(f))$, so that some cylinder set $[w]_j$ enables $s$.
Choose $u = v = 0^{|j|+1}$.
Then every $x \in [u w v]_{j-|u|}$ satisfies $x_{-|j|-1} = 0$, so that $f^{|j| + 1 + |s|}(x)_{[0, |s|)} = 0^{|s|}$.
Since a cell can never change its state from $0$ to $1$, we have $s = 0^{|s|}$.
Hence the language of $\glimset(f)$ is $0^*$, and the claim is proved.
\end{example}


\section{Complexity of generic limit sets}

From the combinatorial characterization we can determine the maximal computational complexity of the language of the generic limit set.

\begin{theorem}
\label{thm:MaxComplexity}
The language of the generic limit set of any CA is $\Sigma^0_3$.
For any $\Sigma^0_3$ set $P$, there exists a cellular automaton $f$ such that $P$ is many-one reducible to $\lang(\glimset(f))$.
\end{theorem}

\begin{proof}
The condition given by Lemma~\ref{lem:CombChar} is $\Sigma^0_3$.

For the second claim, since $P$ is a $\Sigma^0_3$ set, there is a computable predicate $\psi$ such that $P = \{ w \in A^* \;|\; \exists m \, \forall m' \, \exists k \, \psi(w, m, m', k) \}$.
Let $M$ be a Turing machine with state set $Q$, initial state $q_0 \in Q$, two final states $q_f^1, q_f^2 \in Q$, a read-only tape with alphabet $\Gamma_A = A \cup \{\#,\$\}$ and a read-write tape with some tape alphabet $\Gamma$ with special symbol $1 \in \Gamma$.
Both tapes are infinite to the right, and $M$ has only one head that sees the same position of both tapes.
When initialized in state $q_0$, the machine checks that the read-only tape begins with $\# w \# \$^m \#$ for some $w \in A^*$ and $m \geq 0$, and the read-write tape begins with $1^{3 n + 5} {\bot}$ for some $n \geq 0$, halting in state $q_f^1$ if this is not the case.
Then it enumerates $n$ pairs $(m', k) \in \N^2$, starting from $(0,0)$ and moving from $(m', k)$ to $(m'+1,0)$ if $\psi(w, m, m', k)$ holds, and to $(m', k+1)$ otherwise.
If the process ends with $k > 0$, then $M$ halts in state $q^1_f$.
Otherwise it writes $0$s to the $|w|+2$ leftmost cells of the read-write tape, goes to the leftmost cell and halts in state $q^2_f$.
Then $w \in P$ if and only if for some $m \in \N$, the machine halts in state $q_f^2$ for infinitely many choices of $n$; denote this condition by $M(w,m,n)$.

Denote $\Sigma_M = (Q \cup \{{\leftarrow},{\rightarrow}\}) \times \Gamma$ and $\Sigma_0 = \{ B, E, S_1, S_2, S_2', S_3, {\vdash} \}$.
We construct a radius-$3$ CA $f$ on the alphabet $\Sigma = (\Sigma_M \cup \Sigma_0) \times \Gamma_A$ to whose generic limit set $P$ reduces.
We write elements of $\Sigma_M \times \Gamma_A$ as triples $(q, g, a) \in (Q \cup \{{\leftarrow},{\rightarrow}\}) \times \Gamma \times \Gamma_A$.
The first track of $f$ contains elements of $\Sigma_M$, which are used to simulate computations of $M$, and $\Sigma_0$, which perform a geometric process that initialized such simulations.
The element $B$ forms a \emph{background} on which the \emph{signals} $E$, $S_1$, $S_2$, $S_2'$ and $S_3$ travel.
The last track of $\Sigma$ is never modified by $f$, and it serves as the read-only tape of $M$ in the simulation.
We think of $f$ as a non-uniform CA over $\Sigma_M \cup \Sigma_0$ whose local function at each coordinate $i \in \Z$ depends on the element $s \in \Gamma_A$ at $i$.
The automaton $f$ is defined by the following constraints:
\begin{itemize}
\item
  The signal $E$ always travels to the right at speed $2$.
  For $k = 1, 2, 3$, as long as the signal $S_k$ or $S_k'$ has $B$s to its right, it travels to the right at speed $k$.
  The signals $E$, $S_1$ and $S_2$ produce $B$s in their wake, while $S_3$ produces $({\leftarrow},1)$-states and $S_2'$ produces $({\leftarrow},{\bot})$-states.
\item
  When the signals $S_2$ and $S_1$ collide, they produce the four-cell pattern ${\vdash} (q_0, 1) ({\leftarrow}, 1) S_3$, where $S_3$ lies at the point of their collision.
  When $S_3$ and $S_2$ collide, they are replaced by an $S_2'$.
\item
  In an interval of the form
  \[
    {\vdash} ({\rightarrow}, g_0) \ldots ({\rightarrow}, g_{m-1}) (q, g_m) ({\leftarrow}, g_{m+1}) \ldots ({\leftarrow}, g_{m+n})
  \]
  that is either unbounded or terminated on its right by $S_3$ or $S_2'$, $f$ simulates a computation of $M$ using $q$ as the head, the $\Gamma$-track as the read-write tape and the $\Gamma_A$-track as the read-only tape.
  If $q = q^i_f$ is a final state, it is replaced by $E$ instead.
\item
  Any pattern not mentioned above produces $E$-states.
\end{itemize}
In particular, the signals $S_1$ and $S_2$ are never created, so they always originate from the initial configuration.
The signal $S_3$ and all Turing machine heads originate either from the initial configuration or a collision of $S_2$ and $S_1$, and $S_2'$ originates from the initial configuration or a collision of $S_3$ and $S_2$.
An $E$-signal, once created, cannot be destroyed.

For a word $w \in A^+$, define $\hat w = (q^2_f {\leftarrow}^{|w|+1}, 0^{|w|+2}, \# w \#)$.
We claim that $w \in P$ if and only if $\hat w \in \lang(\glimset(f))$.
The proof is visualized in Figure~\ref{fig:MaxComplexity}.

\begin{figure}[ht]
\begin{center}
\begin{tikzpicture}

\draw (-4.5,0) -- (5,0);
\foreach \x in {-0.7,0,2,3}{
  \draw (\x,-0.1) -- ++(0,0.2);
}
\foreach \x in {-1.2,-1.4,-1.6,-3.2}{
  \fill (\x,0) circle (0.05cm);
}

\draw (-1.6,0) -- (0,3.2) -- (-3.2,0);
\draw (-1.2,0) -- ++(6.6,6.6);
\draw (-1.4,0) -- ++(5,5) -- (0,3.2) -- (0,6.6);
\draw (3.6,5) -- ++(1.6,1.6);
\draw [dashed] (3.6,5) -- ++(0,1.6);
\draw [dashed] (0,0) -- (0,3.2);

\draw [densely dotted] (0,3.2)
\foreach \dx in {1,1,0,1,1,0,-1,1,1,1,0,1,0,1,0,-1,1}{
  -- ++(\dx*0.1,0.2)
}
;

\draw [fill=white] (0,3.2) circle (0.07cm);

\node [below right] at (-1.2,0) {$E$};
\node [below] at (-1.4,0) {$S_2$};
\node [below left] at (-1.6,0) {$S_1$};
\node [below left] at (-3.2,0) {$S_2$};
\node [above left] at (1.8,4) {$S_3$};

\node [above] at (-0.35,0) {$u$};
\node [above] at (1,0) {$\tilde w$};
\node [above] at (2.5,0) {$v$};

\node at (0.6,2.7) {$B$};
\node at (-2,0.4) {$B$};
\node at (2,5.5) {$({\leftarrow},1)$};
\node at (4.25,6.3) {$({\leftarrow},{\bot})$};

\node [above] at (0,6.6) {$\vdash$};
\node [above] at (0.8,6.6) {$q$};
\node [above] at (5.2,6.6) {$S_2'$};
\node [above right] at (5.4,6.6) {$E$};

\end{tikzpicture}
\end{center}
\caption{Proof of Theorem~\ref{thm:MaxComplexity}, not drawn to scale. Time increases upward.}
\label{fig:MaxComplexity}
\end{figure}
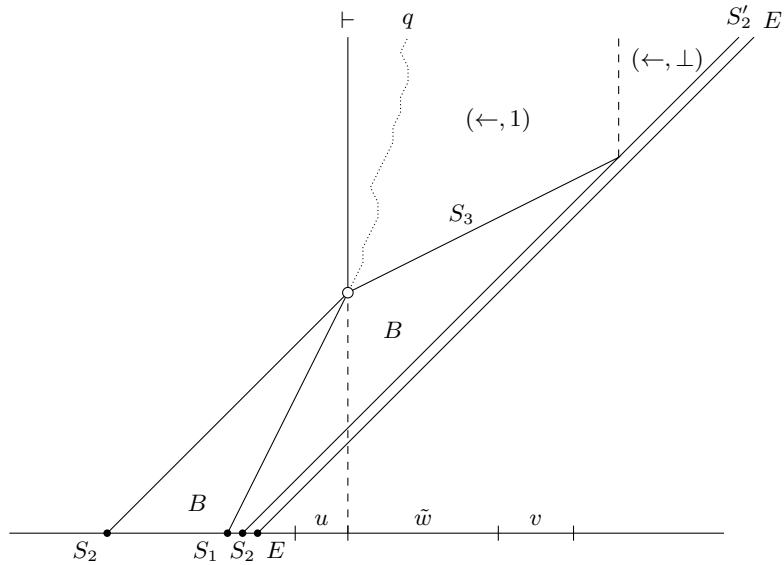

Suppose first that $w \in P$, so that there exists $m \in \N$ such that $M(w,m,n)$ holds for infinitely many $n$.
We claim that the word $\tilde w = (E^{|w|+m+3}, \# w \# \$^m \#)$ enables $\hat w$, so let $u, v \in \Sigma^*$ be arbitrary.
We construct a configuration $x \in [u \tilde w v]_{-|u|}$, which corresponds to the horizontal line of Figure~\ref{fig:MaxComplexity}, as follows.
To the right of $v$ we put $(E, \#)^\infty$, and to the left of $u$ we put only $\#$-symbols on the second track.
Let $n > |u| + 4$ be such that $M(w, m, n)$ holds.
On the first track of $x_{[-n+2,-n+4]}$, put $S_1 S_2 E$; on $x_{-2n+2}$, put $S_2$; on all remaining cells put $B$.
The $E$-signals will destroy everything in their path and replace them with $B$-cells, so we can ignore the contents of the first track of $x_{[-|u|, \infty)}$.
The $S_1$-signal and the leftmost $S_2$-signal will collide at coordinate $2$ of $f^n(x)$ (the white circle in Figure~\ref{fig:MaxComplexity}), resulting in the pattern ${\vdash} (q_0,1) ({\leftarrow}, 1) S_3$ at coordinate $-1$ and $B$s to its left.
The simulated computation of $M$ begins at this time step.
The resulting $S_3$-signal collides with the rightmost $S_2$-signal at coordinate $3 n + 5$ of $f^{2 n + 1}(x)$, producing $({\leftarrow},1)$-states until that point and transforming into a $S_2'$ that produces $({\leftarrow},{\bot})$-states.
This means $M$ has $1^{3n+5} {\bot}$ on its the read-write tape at the beginning of the computation, and $\# w \# \$^m \#$ on the read-only tape.
Hence it eventually writes $0^{|w|+2}$ to the tape and halts in state $q^2_f$ at some time step $t \in \N$.
Then $f^t(x) \in [\hat w]$, and we have showed that $\tilde w$ enables $\hat w$, so $\hat w \in \lang(\glimset(f))$ by Lemma~\ref{lem:CombChar}.

Suppose then $\hat w \in \lang(\glimset(f))$, so that $\hat w$ is enabled by some word $w' \in \Sigma^*$ at coordinate $i \in \Z$.
We may assume $i \leq 0$ and $|w'| \geq i + |w| + 2$ by extending $w'$ if needed.
Let $k \geq 0$ and choose $u = (E, \#)^k$ and $v = (E, \#)$.
In a configuration $x$ that contains $u w' v$, any Turing machine head to the right of $u$ is eventually erased by the $E$-symbols.
Those within $w'$ are erased after $|w'|$ steps, and those to the right of $w'$ are erased before they reach the origin.
Thus, if $f^t(x) \in [\hat w]$ for some $t > |w'|$, then the $q^2_f$ in this configuration is the head of a Turing machine produced at the origin by a collision of some $S_2$-signal and $S_1$-signal at some earlier time $t' < t$ (again the white circle in Figure~\ref{fig:MaxComplexity}).
After a finite computation, $M$ can halt in state $q^2_f$ only at the left end of the tape, so the collision happens at coordinate $2$ and $t' > |w'|$.
Since the signals $S_2$ and $S_1$ cannot be created, they originate at coordinates $-2 t' - 2$ and $- t' - 2$ of $x$.
Since the Turing machine eventually halts in state $q^2_f$, after being initialized it will read $1^{3n+5} {\bot}$ on its read-write tape and $\# w \# \$^m \#$ on the read-only tape for some $m, n \in \N$ with $M(w, m, n)$.
Since the read-only tape cannot be modified by $f$, $w'$ already contains the word $\# w \#$ on its second track, and since $v$ has $\#$ on its second track, $w'$ must contain $\# w \# \$^m \#^p$ for some $p \geq 0$.
Hence $m$ is independent of $k$.

The signal $S_3$ produced at the same collision as $q_0$ continues to the right at speed $3$, producing $({\leftarrow},1)$-states until it is destroyed.
To its right we have $B$-states produced by the initial $E$-signals in $u$, followed by those $E$-signals.
Since the Turing machine reads $1^{3n+5} {\bot}$ on its tape, the $S_3$-signal is destroyed after $n + 1$ steps, at time $t' + n + 1$, either by encountering an invalid pattern or by collision with an $E$-signal or $S_2$-signal.
In the first two cases, after the removal of $S_3$ the segment of $({\leftarrow},1)$-states produced by it is now bordered by an $E$-state, which is an invalid pattern and results in new $E$-states by the last rule of $f$.
These $E$-states will eventually destroy the entire computation segment before the Turing machine can halt.
Hence $S_3$ must collide with an $S_2$-signal at coordinate $3 n + 5$ at time $t' + n + 1$.
This signal originates at position $-2 t' + n + 3$ in $x$, which must be to the right of the $S_1$-signal at coordinate $-t' - 2$ that produces $S_3$, since these signals do not collide.
Hence $n > t' - 5 > k + |w'| - 5$, so $n$ grows arbitrarily large with $k$.
We have shown $w \in P$.
\end{proof}

If we know more about the structure of $\glimset(f)$ and the dynamics of $f$ on it, we can improve the computability bound.

\begin{proposition}
\label{prop:Minimal}
Let $f$ be a CA.
If $\glimset(f)$ is a minimal subshift, then its language is $\Sigma^0_2$.
\end{proposition}

\begin{proof}
Denote $X = \glimset(f)$ and let $w \in \lang(X)$.
Since $X$ is minimal, there exists $n \in \N$ such that $w$ occurs in each word of $\lang_n(X)$.
Let $[v]_j$ be an $X$-forcing cylinder set that satisfies $f^t([v]_j) \subset [\lang_n(X)]$ for all large enough $t$, as given by Lemma~\ref{lem:ForcingWords}.
For these $t$, the set $f^t([v]_j)$ intersects $[w]_i$ for some $0 \leq i \leq n - |w|$.
On the other hand, if $w \notin \lang(X)$ then such a word $v$ does not exist, since each word can be extended into one that eventually forbids $w$.
This means that $w \in \lang(X)$ is equivalent to the $\Sigma^0_2$ condition that there exist $v \in A^*$, $j \in \Z$, $n \in \N$ and $T \in \N$ such that for all $t \geq T$ we have $f^t([v]_j) \cap \bigcup_{i=0}^{n-|w|} [w]_i \neq \emptyset$.
\end{proof}


\begin{proposition}
\label{prop:EquicontSigma01}
Let $f$ be a CA and suppose that its restriction to $\glimset(f)$ is equicontinuous.
Then $\glimset(f)$ has a $\Sigma^0_1$ language.
\end{proposition}

\begin{proof}
Denote $X = \glimset(f)$.
An equicontinuous CA on any subshift is eventually periodic (this was shown in~\cite{Ku97} for the full shift, and the general case is not much more difficult), so that there exist $k \geq 0, p \geq 1$ with $f^{k+p}|_X = f^k|_X$.
Let $r \in \N$ be a common radius of $f$ and $f^p$, and let $[w]_j \subset A^\Z$ and $T \in \N$ be given by Lemma~\ref{lem:ForcingWords}, so that $f^t([w]_j) \subset [\lang_{3 r}(X)]$ for all $t \geq T$.
By extending $w$ if necessary, we may assume $j \leq 0$, $|w| = 3 r + 2 h$ and $f^t(x)_{[r, 2r)} = f^t(x')_{[r, 2r)}$ for all $x, x' \in [w]_j$ and $t \in [T, T+k+p)$, where $h = |j|$.
Since $f$ has radius $r$ and is eventually periodic on $X$, we then have $v_t := f^t(x)_{[r, 2r)} = f^t(x')_{[r, 2r)}$ for all $t \geq T$, and the sequence of words $(v_t)_{t \geq T+k}$ is $p$-periodic.



Let $n \geq 0$ and $u \in A^{2 n}$. 
For $t \geq T$ we have $f^t([w u w]_{-|w|-n}) \subset [v_t]_{-2r-h-n} \cap [v_t]_{n+h+r}$, so that no information can be transmitted over the $v_t$-words.
For $x \in [w u w]_{-|w|-n}$ the sequence of words $s = (f^t(x)_{[-2r-h-n, n+h+2r)})_{t \geq T}$ only depends on its values at $t \in [T, T+k+p)$, and eventually contains only words of $\lang(X)$ since we may extend the central pattern of $x$ into one that is $X$-forcing.
Thus $s$ is eventually $p$-periodic.
Since each word $s_{t+1}$ is determined by $s_t$ using the local rule of $f$, the eventually periodic part is reached when a repetition occurs.
The prefixes and suffixes of length $r$ of each word $s_t$ already form $p$-periodic sequences (since they are equal to $v_t)$, so this happens after at most $p|A|^{2 (r + h + n)}$ steps.

Let $v \in A^*$ be arbitrary.
By Lemma~\ref{lem:CombChar}, $v \in \lang(X)$ if and only if there is a cylinder set $[v']_i$ with $f^t([u' v' w']_{i-|u'|}) \cap [v] \neq \emptyset$ for all $u', w' \in A^*$ and infinitely many $t$.
By extending $v'$ if necessary, we may assume $[v']_i = [w u w]_{-|w|-n}$ for some $n \geq |v|$ and $u \in A^{2 n}$.
Then $v$ occurs infinitely often in words of the eventually periodic sequence $s$.
We have shown that $v \in \lang(X)$ if and only if there exist $n \geq |v|$ and $u \in A^{2 n}$ with $f^t([w u w]_{-|w|-n}) \cap [v] \neq \emptyset$ for some $t \in [T, T+p|A|^{2 (r + h + n)}+k+p)$, where $w$, $T$ and $h$ are fixed.
Hence $X$ has a $\Sigma^0_1$ language.
\end{proof}

\begin{proposition}
  \label{prop:Mixing}
  If a CA $f$ is the identity on $\glimset(f)$, then $\glimset(f)$ is a mixing subshift.
\end{proposition}

\begin{proof}
  Let $v_1, v_2 \in \lang(X)$ be arbitrary, and let $[w_1]_{i_1}, [w_2]_{i_2}$ be two cylinder sets that enable them and are $X$-forcing in the sense that $f^t([w_j]_{i_j}) \subset [\lang_{n + 2r}]_{-r}$ for all $t \geq T$.
  We may assume, by extending the $w_j$ and increasing $T$ if necessary, that $f^T(x)_{[0, n)} = v_j$ for all $x \in [w_j]_{i_j}$, and then $f^t([w_j]_{i_j}) \subset [v_j]$ for all $t \geq T$.
  For all large enough $N$ the intersection $[w_1]_{j_1} \cap [w_2]_{j_2 + N}$ is nonempty, and hence contains an $X$-forcing cylinder $[u]_k$ with $f^t([u]_k) \subset [\lang_{N + |v_2|}(X)]$ for all large enough $t$.
  This implies that $v_1 A^{N-|v_2|} v_2$ intersects $\lang(X)$ for all large enough $N$, i.e. $X$ is mixing.
\end{proof}

We say that a CA $f : X \to X$ on a subshift $X \subset A^\Z$ is \emph{eventually oblique} if $f^n$ has a neighborhood that is contained in $(-\infty, -1]$ or $[1, \infty)$ for some (equivalently, all large enough) $n \in \N$.
All shift maps except the identity are eventually oblique.

\begin{proposition}
\label{prop:Pi02Bound}
Let $f$ be a CA and suppose that its restriction to $\glimset(f)$ is eventually oblique.
Then $\glimset(f)$ has a $\Pi^0_2$ language.
\end{proposition}


\begin{proof}
Denote $X = \glimset(f)$ and let $w \in A^*$.
We claim that $w \in \lang(X)$ if and only if the empty word enables $w$ for $f$, which is a $\Pi^0_2$ condition.
By Lemma~\ref{lem:CombChar} it suffices to prove the forward direction, and the idea of the proof is the following.
Since $w$ occurs in $X$, it is enabled by some cylinder set, which we can extend into one that eventually forces a long segment to contain patterns of $X$.
On this segment, information can flow only from right to left under iteration of $f$.
We use another $X$-forcing cylinder to block all information flow from the enabling word to the right-hand side of this segment before it is formed.
Then the contents of the segment are independent of the word that originally enabled $w$, so we can swap it for any other word.
The argument is visualized in Figure~\ref{fig:Pi02Bound}.

\begin{figure}[ht]
\begin{center}
\begin{tikzpicture}

\draw (-3,0) -- (3.5,0);
\foreach \x in {-2,-1,0,2,3}{
  \draw (\x,-0.1) -- ++(0,0.2);
}
\draw [fill=black!25] (1,5) -- (1,0.5) -- (0.5,0.5) -- (0.5,2) -- (-1.5,2) -- (-1.5,5);
\draw [dashed] (0,0) -- (0.5,0.5);
\draw [dashed] (0.5,2) -- ++(-2,2);
\draw (-0.75,4.5) -- (-0.25,4.5);
\draw (-0.75,4.4) -- ++(0,0.2);
\draw (-0.25,4.4) -- ++(0,0.2);

\draw [dotted] (1,0.5) -- (3.5,0.5);
\draw [dotted] (0.5,2) -- (3.5,2);
\draw [dotted] (-1.5,4) -- (3.5,4);
\draw [dotted] (-1.5,4.5) -- (3.5,4.5);

\node [below] (ca) at (-1.5,0) {$\strut c a$};
\node [below,baseline=(ca.base)] at (-0.5,0) {$\strut u/u'$};
\node [below,baseline=(ca.base)] at (1,0) {$\strut v$};
\node [below,baseline=(ca.base)] at (2.5,0) {$\strut b d$};
\node [below,baseline=(ca.base)] at (-0.5,4.5) {$\strut w$};

\node [above=0.2cm] (j) at (-1,0) {$\strut j$};
\node [above=0.2cm,baseline=(j.base)] at (0,0) {$\strut k$};

\node [right] at (3.5,0.5) {$n T$};
\node [right] at (3.5,2) {$n T'$};
\node [right] at (3.5,4) {$K$};
\node [right] at (3.5,4.5) {$t$};

\end{tikzpicture}
\end{center}
\caption{Proof of Proposition~\ref{prop:Pi02Bound}, not drawn to scale. Time increases upward. In the shaded region, information flows only from right to left. The configurations $f^s(x)$ and $f^s(y)$ agree on the part that is right of the dashed line.}
\label{fig:Pi02Bound}
\end{figure}
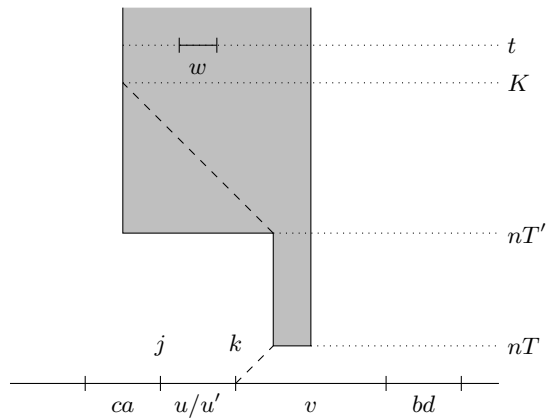

We assume without loss of generality that some $f^n$ has $[1, r]$ as a neighborhood on $X$, where $r \in \N$ is a common radius for each $f^h$ on $A^\Z$ for $0 \leq h \leq n$.
Lemma~\ref{lem:ForcingWords} gives us a cylinder set $[v]_i \subset A^\Z$ and $T \in \N$ with $f^t([v]_i) \subset [\lang_{3 r}(X)]$ for all $t \geq n T$.
By extending $v$ if necessary, we may assume $-|v| \leq i \leq -r n T$.

Assume $w \in \lang(X)$, so that there exists a cylinder set $[u]_j \subset A^\Z$ that enables it for $f$.
For the empty word to enable $w$, it suffices to show that for an arbitrary cylinder set $[u']_{j'}$ there exist infinitely many $t \in \N$ with $f^t([u']_{j'}) \cap [w] \neq \emptyset$.
By extending $u$ and/or $u'$ if necessary, we may assume $j = j'$ and $|u| = |u'| \geq |j|$, and denote $k = j + |u|$.
Consider the cylinder set $[u v]_j$.
By Lemma~\ref{lem:ForcingWords}, there exists a cylinder set $[a u v b]_{j - |a|}$ with $f^t([a u v b]_{j - |a|}) \subset [\lang_{3r+k-i}(X)]_{-2r}$ for all large enough $t \in \N$.
For the same reason, there exists another cylinder set $[c a u' v b d]_{j - |c a|}$ with $f^t([c a u' v b d]_{j - |c a|}) \subset [\lang_{3r+k-i}(X)]_{-2r}$ for all large enough $t \in \N$.
Let $T' \geq T$ be such that $n T'$ is a common bound for these conditions.
Denote $K = n T' + n |u v b d| - j$, and let $t \geq K$ be such that $f^t(x) \in [w]$ for some $x \in [c a u v b d]_{j - |c a|}$.
There are infinitely many such $t$ since $u$ enables $w$ for $f$.
Let $y \in [c a u' v b d]_{j - |c a|}$ be the configuration obtained by replacing the $u$ in $x$ by $u'$.

We claim $f^t(y) \in [w]$.
Note first that $f^{h n}(x)_{[k+s r, \infty)} = f^{h n}(y)_{[k+s r,\infty)}$ for all $h \leq T$ since this holds for $h = 0$ and $r$ is a radius for $f^n$.
This is represented in Figure~\ref{fig:Pi02Bound} by the lower dashed line.
Since $x, y \in [v]_k$, for all $s \geq n T$ we have $f^s(x), f^s(y) \in [\lang_{3 r}(X)]_{k-i}$.
From $i \leq -r n T$ it follows that $k - i \geq k + r n T$, so that $f^{n T}(x)$ and $f^{n T}(y)$ agree on $[k-i, \infty)$.
Since $f^n$ has $[-r, r]$ as a neighborhood on $A^\Z$ and $[1,r]$ as a neighborhood on $X$, for each $j \in [k-i+r, k-i+2r]$ the value of $f^{s+n}(z)_\ell$ for $z \in \{x,y\}$ depends only on $f^s(z)_{[\ell+1, \ell+r]}$.
Thus if $f^s(x)$ and $f^s(y)$ agree on $[k-i+r, \infty)$ for some $s \geq n T$, then so do $f^{s+n}(x)$ and $f^{s+n}(y)$.
By induction, we obtain $f^{n s}(x)_{[k-i+r, \infty)} = f^{n s}(y)_{[k-i+r, \infty)}$ for all $s \geq T$.

Denote $g(s) = \max(-r, k-i+r-s)$.
Then $f^{n (T' + s)}(x)$ and $f^{n(T' + s)}(y)$ agree on $[g(s), \infty)$ for all $s \geq 0$.
This is represented in Figure~\ref{fig:Pi02Bound} by the upper dashed line.
For $s = 0$ this is true by the previous paragraph, so suppose it holds for some $s \geq 0$.
Since $x \in [a u v b]_{j-|a|}$ and $y \in [c a u' v b d]_{j - |c a|}$, we have $f^{n (T' + s)}(x), f^{n (T' + s)}(y) \in [\lang_{3r+k-i}(X)]_{-2r}$.
As in the previous paragraph, the value of $f^{n (T' + s + 1)}(z)_\ell$ for $\ell \in [g(s)-1, k-i]$ and $z \in \{x,y\}$ depends only on $f^{n(T' + s)}(z)_{[\ell+1, \ell+r]}$.
The claim follows by induction.

Writing $t = n s + h$ for $0 \leq h < n$, the configurations $f^t(x)$ and $f^t(y)$ agree on $[g(s)+r, \infty) = [0, \infty)$ since $r$ is a radius for $f^h$.
In particular $f^t(y) \in [w]$, so that $f^t([u']_j) \cap [w] \neq \emptyset$.
\end{proof}

\begin{proposition}
  \label{prop:ChainTrans}
  If the restriction of a CA $f$ to $\glimset(f)$ is a shift map, then $\glimset(f)$ is a chain transitive subshift.
\end{proposition}

\begin{proof}
Suppose $f|_X = \sigma^n|_X$ for some $n \in \Z$.
By symmetry we may assume $n > 0$.
Let $v_1, v_2 \in \lang(X)$ be two words of equal length $m$, and let $N \geq m$ be arbitrary.
We claim that there exist $M > 0$ and words $u_0, \ldots, u_M \in \lang_{N+n}(X)$ such that $v_1$ is a prefix of $u_0$, $v_2$ is a prefix of $u_M$ and the length-$N$ suffix of each $u_i$ is a prefix of $u_{i+1}$, which implies the chain transitivity of $X$.
For this, let $[w]_j$ be an $X$-forcing cylinder with $f^t([w]_j) \subset [\lang_{2 r + N + n}(X)]_{-r}$ for all large enough $t$.
By the proof of Proposition~\ref{prop:Pi02Bound}, $[w]_j$ enables both $v_1$ and $v_2$, so there exist $x \in [w]_j$ and $T \leq t_1 < t_2$ with $f^{t_i}(x)_{[0, m)} = v_i$ for $i = 1, 2$.
Choose $M = t_2 - t_1$ and $u_k = f^{t_1 + k}(x)_{[0, N+n)}$.
Since $r$ is a radius for $f$, these words have the required properties.
\end{proof}

Using these results, we can prove that some individual subshifts cannot occur as generic limit sets.

\begin{example}
  \label{ex:0011}
There is no CA $f : A^\Z \to A^\Z$ whose generic limit set is the orbit closure of ${}^\infty 0 1^\infty$.
Suppose for a contradiction that there is one.
Since $X = \glimset(f)$ is invariant under $f$, we have $f({}^\infty 0 1^\infty) = \sigma^n({}^\infty 0 1^\infty)$ for some $n \in \Z$, and then $f|_X = \sigma^n|_X$.
If $n = 0$, then Proposition~\ref{prop:Mixing} implies that $X$ is mixing, and if $n \neq 0$, then Proposition~\ref{prop:ChainTrans} implies that $X$ is chain transitive, but it is neither.
\end{example}

By the results of~\cite{Ol13}, all cellular automata on Sturmian shifts are restrictions of shift maps.
Propositions~\ref{prop:EquicontSigma01} and~\ref{prop:Pi02Bound} imply that the language of a Sturmian generic limit set is $\Pi^0_2$, and Proposition~\ref{prop:Minimal} (or the folklore result that every minimal $\Pi^0_n$ subshift is $\Sigma^0_n$) implies that it is $\Sigma^0_2$.
Hence we obtain the following.

\begin{corollary}
\label{cor:SturmianUpperBound}
If a Sturmian shift is the generic limit set of a CA, then its language is $\Delta^0_2$.
\end{corollary}

\section{Periodic factors}

In some situations, a nontrivial finite factor forbids a subshift from being realized as a generic limit set.

\begin{definition}
Let $X \subset A^\Z$ be a subshift.
The \emph{chain relation of width $n$} is the relation on $\lang_n(X)$ defined by $u \sim_n v$ if there exists $x \in X$ with $x_{[0, n)} = u$ and $x_{[k, k+n)} = v$ for some $k \geq 0$.
The symmetric and transitive closure of ${\sim_n}$ is the \emph{$\sigma^\pm$-chain relation of width $n$}.
If each ${\sim_n}$ is equal to $\lang_n(X)^2$, we say $X$ is \emph{$\sigma^\pm$-chain transitive}.
A \emph{$\sigma^\pm$-chain component} of $X$ is a maximal $\sigma^\pm$-chain transitive subshift of $X$.
\end{definition}

It is not hard to see that every subshift is the union of its $\sigma^\pm$-chain components, which are disjoint.
SFTs and sofic shifts have a finite number of such components, but in other cases their number may be infinite.

\begin{example}
  Let $X \subset \{0,1,2\}^\Z$ be the union of the orbit closures of ${}^\infty 0 2^\infty$ and ${}^\infty 1 2^\infty$.
  For each $n \in \N$, we have $0^n 2^n, 1^n 2^n \in \lang(X)$, which implies $0^p 2^{n-p} \sim_n 2^n$ and $1^p 2^{n-p} \sim_n 2^n$ for all $0 \leq p \leq n$.
  Since $\lang_n(X)$ consists of exactly these words, $X$ is $\sigma^\pm$-chain transitive.
\end{example}

\begin{lemma}
\label{lem:ChainComps}
Let $f$ be a CA on $A^\Z$ such that $\glimset(f)$ has a finite number of $\sigma^{\pm 1}$-chain components $X_1, \ldots, X_k$.
Then there is a cyclic permutation $\rho$ of $\{1, \ldots, k\}$ such that $f(X_i) = X_{\rho(i)}$ for each $i = 1, \ldots, k$.
\end{lemma}

\begin{proof}
Since the image of a $\sigma^{\pm 1}$-chain transitive subshift by a cellular automaton is also $\sigma^{\pm 1}$-chain transitive, each $X_i$ is mapped into some other component $X_{\rho(i)}$.
This defines a function $\rho : \{1, \ldots, k\} \to \{1, \ldots, k\}$.
By Lemma~\ref{lem:Invariant}, $\rho$ is surjective, hence a permutation, and $f$ maps each $X_i$ surjectively to $X_{\rho(i)}$.
Corollary~4.13 in~\cite{DjGu19} implies that $\rho$ must be a cyclic permutation.
\end{proof}


\begin{proposition}
\label{prop:NoPeriod}
If a subshift $X \subset A^\Z$ has a finite number of $\sigma^{\pm 1}$-chain components and a finite factor that does not consist of fixed points, then it is not the generic limit set of any CA.
\end{proposition}

\begin{proof}
  Suppose that $f$ is a CA with $\glimset(f) = X$, and let $\pi : (X, \sigma) \to (Y, g)$ be the morphism onto a nontrivial finite factor.
Let $X_1, \ldots, X_k$ be the $\sigma^{\pm 1}$-chain components of $X$.
By Lemma~\ref{lem:ChainComps} we may assume $f(X_i) = X_{i+1 \bmod k}$ for all $i$.
If $Y_p \subset Y$ is the subsystem of $p$-periodic points, then $\pi^{-1}(Y_p) \subset X$ is an $f$-invariant subshift consisting of $\sigma^{\pm 1}$-chain components, so it is nonempty for exactly one $p$, and $p > 1$ by the assumption that $Y$ does not consist of fixed points.
By taking a factor map from $Y$ onto $\Z_p$ if necessary, we may assume $(Y, g) = (\Z_p, {+1})$ where the addition is modulo $p$.
Define $q : X \to \Z_p$ by $q(x) = \pi(x) - \pi(f(x))$.
Then $q$ is continuous and shift-invariant, and is constant in each component $X_i$.
Let $r$ be a common radius for $f$ and right radius for $\pi$ and $q$, meaning that $\pi(x)$ and $q(x)$ are determined by $x_{[0,r)}$.
The right radii exist since $\pi$ has some two-sided radius $s$ by continuity and satisfies $\pi(x) = \pi(\sigma^s(x)) - s$, which is determined by $x_{[0, 2s]}$, and similarly for $q$.
For $w \in \lang_{2r+1}(X)$, denote $\pi(w) = \pi(x)$ and $q(w) = q(x)$ for any $x \in [w]_{-r}$.

Let $u \in \lang_{3 r}(X_1)$ be arbitrary, let $[v]_i$ be a cylinder that enables it given by Lemma~\ref{lem:CombChar}, and let $[w]_j \subset [v]_i$ be an $X$-forcing cylinder given by Lemma~\ref{lem:ForcingWords}, so that $f^t([w]_j) \subset [\lang_{3 r}(X)]$ for all $t \geq T$.
We may assume that $f^T([w]_j) \subset [u]$ by extending $w$ if necessary.
Let $m \in \N$ be such that $[w]_j \cap [w]_{j + p m + 1}$ is nonempty, which holds for all large enough $m$.
Finally, let $[v']_\ell \subset [w]_j \cap [w]_{j + p m + 1}$ be an $X$-forcing cylinder with $f^t([v']_\ell) \subset [\lang_{3 r + p m + 1}(X)]$ for all large enough $t$.

Let $x \in [v']_\ell$ be arbitrary.
Then we have $f^T(x) \in [u] \cap [u]_{p m + 1}$ and $f^t(x) \in [\lang_{3 r}(X)] \cap [\lang_{3 r}(X)]_{p m + 1}$ for all $t \geq T$.
This implies $\pi(f^{t+1}(x)_{[r, 2r)}) = \pi(f^t(x)_{[r, 2r)}) + q(f^t(x)_{[r, 2r)})$ and $\pi(f^{t+1}(x)_{p m + 1 + [r, 2r)}) = \pi(f^t(x)_{p m + 1 + [r, 2r)}) + q(f^t(x)_{p m + 1 + [r, 2r)})$ for each $t \geq T$.
Since $u \in \lang_{3 r}(X_1)$, $f$ permutes the components $X_i$ and $q$ is constant in each component, we have $\pi(f^t(x)_{[r, 2r)}) = \pi(f^t(x)_{p m + 1 + [r, 2r)})$ for $t \geq T$.
For large enough $t$, this is a contradiction with the fact that $w^{(t)} = f^t(x)_{[0, p m + 3 r]} \in \lang(X)$ satisfies $\pi(w^{(t)}_{[i, i+r)}) = \pi(w^{(t)}_{[0, r)}) + i$ for all $i \in [0, p m + 2 r]$.
\end{proof}

As a corollary, a transitive but nonmixing SFT cannot be the generic limit set of a CA, since it admits a factor map to some $\Z_p$ with $p > 1$.

\section{Future work}

In this article we presented a construction for a generic limit set with a maximally complex language.
We also showed that structural and dynamical properties may constrain this complexity, but did not prove the strictness of these bounds.
Furthermore, we showed that in some cases a global period is an obstruction for being a generic limit set, but did not prove whether this always holds.
Using a construction technique presented in~\cite{BoDePoSaTh15} to study $\mu$-limit sets, we believe we can answer these and other questions.

A cellular automaton $f$ is nilpotent, meaning that $f^t(A^\Z)$ is a singleton set for some $t \in \N$, if and only if its limit set is a singleton~\cite{CuPaYu89}.
As a variant of nilpotency, one may define $f$ to be \emph{generically nilpotent} if $\glimset(f)$ is a singleton, as in Example~\ref{ex:GenNilp}.
We believe the techniques of~\cite{BoDePoSaTh15} can also be used to characterize the complexity of deciding generic nilpotency.


\bibliographystyle{splncs04}
\bibliography{GenericsBib}

\end{document}